\documentclass[12pt]{amsart}

\usepackage{amscd} 
\usepackage{mathrsfs}

\newtheorem{theorem}{Theorem}[section]
\newtheorem{lemma}[theorem]{Lemma}
\newtheorem{proposition}[theorem]{Proposition}
\newtheorem{corollary}[theorem]{Corollary} 
\theoremstyle{definition}

\newtheorem{remark}[theorem]{Remark} 
\numberwithin{equation}{section}

\def\A{\mathcal A}
\def\H{\mathscr H}
\def\fh{\mathfrak h}

\def\N{\mathscr N}
\def\D{\mathscr D}
\def\RE{\mathbb R}
\def\C{{\mathbb C}}

\def\t {\tilde}

\begin{document}

\title[Nonlinear extensions of symmetric operators ]
{Nonlinear Maximal Monotone extensions of symmetric operators}

\author{Andrea Posilicano}

\address{DiSAT, Universit\`a dell'Insubria, I-22100
Como, Italy}

\email{posilicano@uninsubria.it}

\begin{abstract}
Given a linear semi-bounded symmetric operator $S\ge -\omega$, we explicitly
define, and provide their nonlinear resolvents, nonlinear maximal monotone
operators $A_\Theta$ of type $\lambda>\omega$  (i.e. generators of one-parameter continuous nonlinear semi-groups of contractions of type $\lambda$) which coincide with the Friedrichs extension of $S$ on a convex set containing the domain of $S$.
The extension parameter $\Theta\subset{\mathfrak h}\times{\mathfrak h}$ ranges over the set of nonlinear
maximal monotone relations in an auxiliary Hilbert space $\mathfrak h$ isomorphic to the deficiency subspace of $S$. 
Moreover $A_{\Theta}+\lambda$ is a sub-potential operator (i.e. is the sub-differential of a
lower semicontinuous convex function) whenever ${\Theta}$ is sub-potential. Applications to Laplacians with nonlinear singular perturbations supported on null sets and to Laplacians with nonlinear boundary conditions on a bounded set are given. 
\end{abstract}

\maketitle
\section{Introduction}
Let $S:\D(S)\subseteq\H\to\H$ be a lower semi-bounded symmetric operator on the Hilbert space $\H$. 
The famed  
Birman-Kre\u\i n-Vishik theory (\cite{[Kre]}, \cite{[Vis]}, \cite{[Bir]}) gives all its lower semi-bounded self-adjoint extensions; here we would like to provide a nonlinear analogue of this theory. First of all we need to define which kind of nonlinear extensions we are looking for. In the linear case, by spectral calculus, we know that the self-adjoint operator $A$ is lower semi-bounded if and only if there exists a real number $\lambda$ such that $e^{-t(A+\lambda)}$, $t\ge 0$, is a continuous semi-group of contractions in $\H$, i.e. $\|e^{-t(A+\lambda)}u\|\le \|u\|$ (equivalently $\|e^{-tA}u-e^{-tA}v\|\le e^{\lambda t}\|u-v\|$). Thus in the nonlinear case we are led to look for nonlinear extensions which are generators of continuous nonlinear semigroups $S_t$, $t\ge 0$, such that $\|S_t(u)-S_t(v)\|\le e^{\lambda t}\|u-v\|$ for some real number $\lambda$.\par   
By the theory of one-parameter continuous nonlinear semi-groups of contractions of type $\lambda$ we know that $S_t$ has a generator given by a monotone operator of type $\lambda$ which is a principal section of a maximal monotone relation (see Section 2 for a compact review of the theory of maximal monotone operators). Since maximal monotonicity can be characterized in terms of nonlinear resolvents and since, in the linear case, the theory of self-adjoint extensions can be formulated in terms of the famed Kre\u\i n's resolvent formula, one is led to look for a nonlinear version of this formula. In Section 3 we show that such a nonlinear generalization can be found and that it gives rise to maximal monotone nonlinear extensions of the symmetric operator $S$ (see Theorem \ref{estensioni} and Remark \ref{char}). It turns out that these nonlinear extensions $A_{\Theta}$ are parametrized by maximal monotone relations $\Theta\subset\fh\times\fh$ in an auxiliary Hilbert space $\fh$ isomorphic to the defect space of $S$. Moreover the nonlinear semigroup $S^{\Theta}_t$ having $A_{\Theta}$ as its generator continuously depends on the linear symmetric operator $S$ and the extension parameter $\Theta$ (see Lemma \ref{conv}).
\par In the linear case to any positive extension one can associate a corresponding bilinear form defined in terms of a positive bilinear form in $\fh$; the quadratic form of the linear extension is a convex lower semicontinuous function and the associated self-adjoint operator is (one half of) its differential. This correspondence has a nonlinear analogue: if the extension parameter is the sub-differential of a convex lower semicontinuous function on $\fh$ then the corresponding extension is the the sub-differential of a convex lower semicontinuous function (see Theorem \ref{cyclic} and Remark \ref{Remark 4.7}). Such a representation in terms of sub-differentials allows for results about the regularity and the asymptotic behavior of the nonlinear semi-groups (see Remarks \ref{Remark 4.5}, \ref{Remark 4.6} and \ref{lr}). \par 
The paper is concluded by Section 5 which contains some applications. In the first one we give a nonlinear version of 
the self-adjoint extensions describing point perturbations of the 3-dimensional Laplacian (see the comprehensive book \cite{[AGHH]} and references therein for the linear case); such an example is then generalized by considering more general singular perturbations of the $n$-dimensional Laplacian supported on $d$-sets. Another application provides maximal monotone realizations of Laplace operators on a bounded regular set with nonlinear boundary conditions; conditions under which such operators generate nonlinear Markovian semigroups are also given.

\section{Preliminaries: Nonlinear Semigroups of Evolution and their Generators}
In this section we briefly recall the basic facts about monotone operators and nonlinear contraction semigroups; we refer to \cite{[Br2]}, \cite{[Paz]}, \cite{[Barb]}, \cite{[Sho]} and references therein for more details and proofs.\par 
Let  $A:\D(A)\subseteq\H\to\H$ be a nonlinear operator on the {\it real} 
Hilbert space $\H$ with scalar product $\langle\cdot  ,\cdot  \rangle$ and
corresponding norm $\|\cdot  \|$. $A$ is said to be {\it monotone  of type $\omega$} ({\it monotone} in case $\omega=0$) if
$$
\forall\,u,v\in \D(A)\,,\quad\langle
(A+\omega)(u)-(A+\omega)(v),u-v\rangle\ge 0 
$$ 
and {\it maximal monotone of type $\omega$} ({\it maximal monotone} in case $\omega=0$) if for some $\lambda>\omega$ (equivalently for any $\lambda >\omega$) one has 
$$
\text{\rm range}\,(A+\lambda)=\H\,.
$$
By such a definition one gets the existence of the nonlinear resolvent: if $A$ is monotone of type $\omega$ then for $\lambda>\omega$
$$
\langle(A+\lambda )(u)-(A+\lambda)(v), u-v\rangle\ge (\lambda-\omega)\|u-v\|^{2}\,.$$ 
Thus if $A$ is maximal monotone of type $\omega$ then
$$
(A+\lambda ):\D(A)\subseteq \H\to\H
$$ 
is bijective for any $\lambda>\omega$ 
and the nonlinear resolvent 
$$
(A+\lambda)^{-1}:\H\to\H\,,\quad \lambda>\omega\,,
$$
is monotone and is a Lipschitz map with Lipschitz constant $(\lambda-\omega)^{-1}$.
\par
Given the nonlinear resolvent $R_{\lambda}:=(A+\lambda)^{-1}$, obviously one has  
$A=R_{\lambda}^{-1}-\lambda$ for any $\lambda>\omega$,
and such a relation is equivalent to the nonlinear resolvent identity 
\begin{equation}\label{nlr}
R_{\lambda}=R_{\mu}\circ(1-(\lambda-\mu)R_{\lambda})\,,
\end{equation}
which holds for any couple $\lambda,\mu\in(\omega,\infty)$. 
Conversely if $R_{\lambda}:\H\to\H$, $\lambda> \omega$, is a family of 
monotone and injective nonlinear maps which satisfies the nonlinear resolvent identity \eqref{nlr},  
then $$A:=(R_{\lambda}^{-1}-\lambda):
\D(A)\subseteq\H\to\H\,,\quad \D(A):=\text{\rm range}(R_{\lambda})\,,$$
is a $\lambda$-independent, maximal monotone nonlinear operator of type $\omega$.
The notion of maximal monotone operator can be generalized by considering multi-valued maps:\par
$\A\subset \H\times \H$ is said to be a {\it monotone relation of type $\omega$} ({\it monotone relation} in case $\omega=0$) if 
$$
\forall\, (u,\t u), (v,\t v)\in \A\,,\quad\langle \t u-\t v,u-v\rangle\ge -\omega  \|u-v\|^{2} 
$$
and is said to be a {\it maximal monotone relation of type $\omega$} ({\it maximal monotone relation} in case $\omega=0$) if it is not properly contained in any other 
monotone relation of type $\omega$. By Minty's theorem (see e.g. \cite[Lecture 3, Theorem 1]{[Paz]}), the graph $$\text{\rm graph}(A):=\{(u,\t u)\in\H\times\H: u\in\D(A),
\ \t u=A(u)\}$$ of a maximal monotone operator of type $\omega$ is a maximal monotone relation of type $\omega$. Conversely, since any $\A\subset\H\times\H$  defines a set-valued  
operator by 
$$ u\mapsto\A(u):=\{\t u\in\H :(u,\t u)\in \A\}$$ 
with domain 
$$\D(\A):=\{u\in\H: \A(u) \not=\emptyset\}$$
and  $\A(u)$ is closed and convex for any maximal monotone relation $\A$ (see e.g. the Lemma in \cite{[Paz]}, Lecture 3), one can associate to a maximal monotone relation $\A\subset\H\times\H$ of type $\omega$ a single-valued nonlinear operator $\A^{0}:\D(\A)\subseteq\H\to\H$ 
by  $\A^{0}(u):=u_{\text{min}}$, where 
$u_{\text{min}}$ is the element of minimum norm in the closed convex set 
$\A(u)$ (see Corollary 2 in \cite{[Paz]}, Lecture 3). By \cite[Corollaire 2.2]{[Br2]}, $\A^{0}_{1}=\A^{0}_{2}\Rightarrow\A_{1}=\A_{2}$,   for any couple of maximal monotone relations. 
\vskip8pt\noindent
In the following we identify a single-valued operator $B$ with its graph; hence, given a relation $\A$,  the writing $\A+B$ means the relation $\A+\text{graph}(B)$.
\vskip8pt\noindent
While the domain of a linear maximal monotone relation is
necessarily dense, in the nonlinear case this can be false; by
Minty-Rockafellar theorem  (see e.g. \cite[Th\'eor\`eme 2.2]{[Br2]}) the closure of the domain $\D(\A)$ of a
maximal monotone relation $\A\subset\H\times\H$ is always a convex set. 
\par 
Given the closed convex nonempty subset ${\mathscr C}\subseteq \H$,  the family of maps $S_t:{\mathscr C} \to{\mathscr C} $, 
$t\ge 0$, is
said to be a one-parameter nonlinear continuous semi-group 
of type $\omega$ (of contractions, in case $\omega=0$) on ${\mathscr C} $  if one has 
$$
S_{0}(u)=u\,,\quad S_{t_1}\!\circ S_{t_2}=S_{t_1+t_2}\,,
\qquad\lim_{t\downarrow 0}\,\|S_t(u)-u\|=0
$$
and
$$
\|S_t(u)-S_t(v)\|\le e^{\omega t}
\|u-v\|\,.
$$
One then defines the generator of the above semigroup by 
$$
A:\D(A)\subseteq \H\to\H\,,\qquad A(u):=\lim_{t\downarrow 0}\, \frac{1}{t}\,(u-S_t(u))\,,
$$ 
where $\D(A)\subseteq {\mathscr C}$ is the set of $u$ such that the above limit exists (by the last Remark in \cite[Lecture 5]{[Paz]} the limits above can be equivalently taken either in strong or in weak sense). The main properties of the semigroup $S_{t}$ and of its generator $A$ are the following ones (see e.g. \cite[Lectures 5 and 6]{[Paz]}):
1) $\D(A)$ is dense in ${\mathscr C} $ and $S_{t}$-invariant;
2) $A$ is monotone of type $\omega$ and there exists a unique maximal monotone relation $\A\subset\H\times\H$ of type $\omega$ such that $A=\A^{0}$; 
3) the path $t\mapsto u(t):=S_t(u)$ is Lipschitz continuous for any $u\in \D(A)$; 
4) $t\mapsto A(u(t))$ is right continuous and $t\mapsto e^{-\omega t}\|A(u(t))\|$ is monotone non-increasing for any $u\in \D(A)$; 5) for any $u\in \D(A)$ and $t>0$, one has  
$$
\frac{d^{+}}{dt}\,u(t)+A(u(t))=0\,, 
$$
where $\frac{d^{+}}{dt}$ denotes the right derivative; 6) for a.e. $t>0$, one has  
$$
\frac{d\,}{dt}\,u(t)+A(u(t))=0\,.
$$
In the linear case ${\mathscr C} =\H$,  
$t\mapsto S_tu$ is continuously differentiable everywhere and by functional calculus a self-adjoint 
operator generates a one-parameter linear continuous semi-group 
of type $\omega$ if and only if $A\ge -\omega$ and $S_t=e^{-tA}$. The nonlinear analogue
of that is given by combining the properties listed above with K\o mura's theorem (see \cite{[Kom]} and  \cite{[Kat]}):  
{\it a maximal monotone operator $A:\D(A)\subseteq \H\to\H$ of type $\omega$  generates a one-parameter nonlinear continuous semi-group 
of type $\omega$ on $\overline{\D(A)}$}.
\section{Nonlinear maximal monotone extensions}
Let $S:\D(S)\subseteq\H\to\H$, $S\ge -\omega$, be a densely defined, semi-bounded symmetric operator. Then $S$ is linear monotone of type $\omega$ but is not maximal monotone since it has 
$A_\circ $ as proper monotone extension, where $A_\circ :\D(A_\circ )\subseteq\H\to\H$, is the linear self-adjoint
operator given by the Friedrichs extension of $S$.\par 
We denote by $\H_{{\circ}}$ the Hilbert space $\D(A_\circ )$ with the
scalar product $\langle\cdot  ,\cdot  \rangle_{{\circ}}$ leading to the graph
norm, i.e.
$$
\langle u,v\rangle_{{\circ}}:=
\langle A_\circ u,A_\circ v\rangle+\langle u,v\rangle\,.
$$
From now on we suppose that $S$ is not essentially self-adjoint; without loss of generality we can take $\bar S=A_{\circ}|\N$, where $\N=\text{\rm 
kernel}\,(\tau)$ is the kernel (which we suppose to be dense in $\H$) of a linear, bounded surjective map
$$
\tau:\H_{{\circ}}\to\fh\,,
$$
onto an auxiliary Hilbert space $\fh$ (with scalar product $[\cdot  ,\cdot  ]$
and corresponding norm $|\cdot  |$) isomorphic to the defect space of $S$ (see e.g. \cite[Section 2.2]{[P12]}).
\par
Our aim here is to construct  {\it nonlinear} maximal monotone operators  $A$ such that 
$$
S\subset A\subset S^{*}\,.
$$ 
Being $\N\subseteq \D(A_\circ )\cap \D(A)$ dense, the
operator $A$ is a {\it nonlinear singular
perturbation of $A_{\circ}$}.\vskip8pt
For any $\lambda>\omega$ we define the bounded linear operators
$$R^{\circ}_\lambda :\H\to\H_\circ\,,
\qquad R^{\circ}_\lambda :=(A_\circ +\lambda)^{-1}
$$
and  
$$G_\lambda :\fh\to\in\H\,,\qquad G_\lambda := (\tau R^{\circ}_\lambda )^*\,.
$$
By the denseness hypothesis on $\N$ one has 
\begin{equation}\label{cap}
\text{\rm range$(G_\lambda )$}\cap
\D(A_\circ )=\left\{0\right\}
\end{equation}
and, by first resolvent identity,
\begin{equation}\label{1st}
(\lambda-\mu)\,R^{\circ}_\mu  G_\lambda =G_\mu -G_\lambda \,,
\end{equation}
i.e.
\begin{equation}\label{2st}
A_{\circ}(G_{\mu}-G_\lambda) =\lambda G_\lambda -\mu G_\mu\,.
\end{equation}
Now, we try  to define a nonlinear extension $A$ by producing 
its nonlinear resolvent $R_\lambda :=(A+\lambda)^{-1}$. Let us write such a presumed resolvent as 
$$
R_\lambda =(1+ \tilde V_\lambda \circ\tau)\circ R^{\circ}_\lambda 
\equiv R^{\circ}_\lambda + \tilde V_\lambda \circ G_\lambda ^{*}\,, 
$$
where the nonlinear map $\tilde V_\lambda :\fh\to \H$ has to be determined. Then, since 
$$
\langle R^{\circ}_\lambda u,u\rangle \ge (\lambda-\omega)\|R^{\circ}_\lambda u\|^{2}\ge 0\,,
$$ 
one has 
\begin{align*}
\langle R_\lambda  (u)-R_\lambda  (v),u-v\rangle
=&
\langle R^{\circ}_\lambda (u-v),u-v\rangle+
\langle \tilde V_\lambda ( G_\lambda ^{*}u)-
\tilde V_\lambda ( G_\lambda ^{*}v),u-v\rangle
\\
\ge 
&
\langle \tilde V_\lambda (G_\lambda ^{*}u)-
\tilde V_\lambda ( G_\lambda ^{*}v),u-v\rangle
\end{align*}
and so, setting $\tilde V_\lambda =G_\lambda V_\lambda$ for some nonlinear   
$V_\lambda:\fh\to\fh$, one gets
\begin{equation*}
\langle R_\lambda  (u)-R_\lambda  (v),u-v\rangle\ge
[V_\lambda (G_{\lambda}^{*}u)-
V_\lambda (G_{\lambda}^{*}v),G_{\lambda}^{*}u-G_{\lambda}^{*}v]
\,.
\end{equation*}
Thus $R_{\lambda}$ is monotone 
whenever 
$$
\forall\,\xi,\zeta\in\fh\,,\quad[V_\lambda(\xi)-
V_\lambda(\zeta),\xi-\zeta]\ge 0\,,
$$ 
namely whenever $V_\lambda$ is monotone.
\par
Suppose now that there 
exists a family of monotone relations $$M_{\lambda}\subset\fh\times\fh\,,\quad 
\lambda>\omega\,,
$$  
such that  
\begin{equation}\label{Gamma2}
\text{$M_{\lambda}-M_{\mu}\ $ is single-valued and}\quad M_{\lambda}-M_{\mu}=(\lambda-\mu)\, G_\mu ^{*}G_{\lambda}
\end{equation}
and
\begin{equation}\label{Gamma3}
Z\not=\emptyset\,,
\end{equation}
where $Z$ is the set of $\lambda>\omega$ such that $\{(\t\xi,\xi):(\xi,\t\xi)\in M_{\lambda}\}$ is the graph of a (necessarily monotone) single-valued map $M_\lambda^{-1}:\fh\to\fh$.
\par Then, setting $V_\lambda=M_{\lambda}^{-1}$, one has the following 
\begin{lemma}\label{res} For any $\lambda\in Z$ let us define 
$$
R_{\lambda}:\H\to\H\,,\quad 
R_{\lambda}=R^{\circ}_{\lambda}+
G_{\lambda}M_{\lambda}^{-1}\circ G_{\lambda}^{*}\,.
$$
Then $R_{\lambda}$ is monotone, injective and satisfies the nonlinear resolvent identity
\begin{equation}\label{nlres}
R_\lambda  =R_\mu \circ(1-(\lambda-\mu)\,R_\lambda  )\,.
\end{equation}
\end{lemma}
\begin{proof} $R_{\lambda}$ is monotone by monotonicity of $M_{\lambda}^{-1}$.\par
Let us now take $u$, $v$ in $\H$ such that $R_{\lambda}u=R_{\lambda}v$. Then
$$
R^{\circ}_{\lambda}(u-v)=-G_{\lambda}(M_{\lambda}^{-1}(G^{*}_{\lambda}u)-M_{\lambda}^{-1}(G^{*}_{\lambda}v))\,.
$$
By \eqref{cap} one gets $u=v$ and therefore $R_{\lambda}$ is injective.\par
By \eqref{1st} and \eqref{Gamma2} one has
\begin{align*}
&R_\mu \circ  (1-(\lambda-\mu)\,R_\lambda  )\\
=&R^{\circ}_\mu   (1-(\lambda-\mu)\,R^{\circ}_\lambda )
-(\lambda-\mu)\,R^{\circ}_\mu  
G_\lambda M_\lambda^{-1}\circ G_\lambda ^{*}
+G_\mu   M_\mu^{-1}\circ  G_\mu ^{*}\\
-&G_\mu   M_\mu^{-1}\circ  (\lambda-\mu)(G_\mu ^{*}  R^{\circ}_\lambda 
+
G_\mu ^{*}G_\lambda   M_\lambda^{-1}\circ   G_\lambda ^{*})\\ 
=&R^{\circ}_\lambda -(G_\mu -G_\lambda )M_\lambda^{-1}\circ  
G_\lambda ^{*}
+G_\mu   M_\mu^{-1}\circ  (G_\lambda ^{*}
-(M_\lambda-M_\mu)  M_\lambda^{-1}\circ  
G_\lambda ^{*})
\\
=
&
R^{\circ}_\lambda +G_\lambda M_\lambda^{-1}\circ  
G_\lambda ^{*}=R_\lambda  \,.
\end{align*}
\end{proof}
By Lemma \ref{res} one immediately gets the following 
\begin{corollary} Let $R_{\lambda}$ be as in Lemma \ref{res} and pose
$$
A:\D(A)\subseteq\H\to\H\,,\quad 
\D(A):=\text{\rm range}(R_{\lambda})\quad A:=R_{\lambda}^{-1}-\lambda\,.
$$
Then $A$ is $\lambda$-independent and maximal monotone of type $\lambda$ 
for any $\lambda\in Z$.
\end{corollary}
As regards the required properties of the family $M_\lambda$,  one has the following
\begin{lemma}\label{lemmagamma} Let $\Theta\subset\fh\times\fh$ be a maximal monotone relation and let $\lambda_\circ >\omega$. Then
\begin{equation}\label{gammateta}
M_\lambda^{\Theta}:=\Theta+
(\lambda-\lambda_\circ )G_{\circ }^{*}G_\lambda \,,\quad \lambda>\omega\,,\quad G_{\circ}:=G_{\lambda_\circ }\,,
\end{equation}
is a maximal monotone relation for any $\lambda\ge \lambda_{\circ}$ and satisfies \eqref{Gamma2} and \eqref{Gamma3} with $(\lambda_\circ ,+\infty)\subseteq Z$; $\lambda_{\circ}\in Z$ whenever $\Theta^{-1}$ is single-valued.
\end{lemma}
\begin{proof} By \eqref{1st}, the family 
$M^\circ _\lambda$, $\lambda>\omega$, of bounded symmetric operators 
$$
M^\circ _\lambda:=\tau(G_{\circ }-G_\lambda )\equiv
(\lambda-\lambda_\circ )\,G_{\circ }^{*}  G_\lambda 
$$
satisfies \eqref{Gamma2}. Thus the relation $M_\lambda^{\Theta}$ satisfies \eqref{Gamma2}. Moreover, by \eqref{1st} again,
\begin{align*}
[M^\circ _{\lambda}\xi,\xi]=&(\lambda-\lambda_\circ )\langle 
G_\lambda \xi,G_{\circ }\xi\rangle
=(\lambda-\lambda_\circ )(\|G_{\circ } \xi\|^{2}-(\lambda-\lambda_\circ ) 
\langle R^{\circ}_{\lambda}G_{\circ } \xi,G_{\circ } \xi\rangle)\\
&\ge(\lambda-\lambda_\circ )\,\frac{\lambda_\circ -\omega}{\lambda-\omega}\,\|G_\circ  \xi\|^{2}\,.
\end{align*}
Since $G_{\circ }^{*}$ is surjective, $G_{\circ }$ has closed range by the closed range theorem and so there exists $\gamma_{0}>0$ such that  
$\|G_{\circ } \xi\|\ge \gamma_{0}\|\xi\|$. Thus $M^{\circ}_{\lambda}$ is monotone of type $-\omega_{0}$ with $\omega_{0}=\gamma_{0}^{2}(\lambda-\lambda_\circ )\,\frac{\lambda_\circ -\omega}{\lambda-\omega}>0$. Moreover, since $\Theta$ is monotone, for any $(\xi,\t\xi)$, $(\zeta,\t\zeta)$ in $\Theta$ one has
\begin{align*}
|(\t\xi+M^\circ _{\lambda}\xi)-(\t\zeta+M^\circ _{\lambda}\zeta)|\,
\ge \omega_{0}|\xi-\zeta|\,,
\end{align*}
and so $(M^\Theta _\lambda)^{-1}(\t\xi):=\{\xi:(\xi,\t\xi)\in M^\Theta _\lambda\}$ is a single-valued map. Since 
$M^\circ _\lambda$ is linear, monotone and bounded, and $\Theta$ is maximal monotone, $M^\Theta _\lambda$ is maximal monotone of type $-\omega_{0}$ by \cite[Lemme 2.4]{[Br2]}. Hence $\D((M^\Theta _\lambda)^{-1})=\fh$ and so $(\lambda_\circ ,+\infty)\subseteq Z$. \par 
Since $\Theta^{-1}$ is maximal monotone (see e.g. \cite[Proposition 2.1]{[Barb]}) 
and $M_{\lambda_{\circ}}^{\circ}=\Theta$, $\lambda_{\circ}\in Z$ whenever $\Theta^{-1}$ is single-valued.
\end{proof}
By collecting the above results finally one gets the following nonlinear version of Kre\u\i n's resolvent formula (see \cite{[P01]}, \cite{[DeMa]} and references therein for the linear case):
\begin{theorem}\label{estensioni}
Given the semi-bounded symmetric operator $S:\D(S)\subseteq\H\to\H$, $S\ge -\,\omega$, the surjective and continuous linear map $\tau:\H_{\circ}\to\fh$, such that $\D(S)=\text{\rm kernel($\tau$)}$, and the maximal monotone relation $\Theta\subset\fh\times\fh$, let   
$\lambda_\circ >\omega$ and define the maximal monotone relation 
$M_{\lambda}^\Theta\subset\fh\times\fh$ as in \eqref{gammateta}.  Then
$$
R^\Theta_\lambda:=R^{\circ}_\lambda +
G_\lambda (M_{\lambda}^\Theta)^{-1}\circ  
G_\lambda^{*} \,,
\qquad \lambda> \lambda_\circ 
$$
is the resolvent of a nonlinear maximal monotone operator $A_\Theta:\D(A_{\Theta})\subseteq\H\to\H$ of type $\lambda_\circ $. Such an operator is
defined by
$$
\D(A_\Theta):=\{u\in\H:u=
u_\circ+G_\circ\xi_{u}\,,\ u_\circ\in \D(A_{\circ})\,,\ (\xi_{u},\tau u_\circ)\in \Theta\}\,,
$$
$$
A_\Theta (u):=A_{\circ}u_\circ-\lambda_{\circ}G_{\circ}\xi_{u}\,.
$$
\end{theorem}
\begin{proof} Defining $A_{\Theta}:=(R^{\Theta}_{\lambda})^{-1}-\lambda$, $\lambda>\lambda_{\circ}$, 
one gets a $\lambda$-independent, 
maximal monotone operator of type $\lambda_\circ$. Thus
$$
\D(A_\Theta):=\{u=
u_\lambda+G_\lambda (M^{\Theta}_{\lambda})^{-1}(\tau u_{\lambda})\}\,,
$$
$$
(A+\lambda)(u)=(A_{\circ}+\lambda)u_{\lambda}\,.
$$
Let us now pose 
$\xi_{u}(\lambda):=(M^{\Theta}_{\lambda})^{-1}\circ\tau u_{\lambda}$, 
so that $u\in\D(A_{\Theta})$ if and only if, for any $\lambda>\omega$, 
$u=u_{\lambda}+G_{\lambda}\xi_{u}(\lambda)$, $u_{\lambda}\in\D(A_{\circ})$ such that  
$$(\xi_{u}(\lambda),\tau u_{\lambda})\in\Theta+\text{graph}((\lambda-\lambda_{\circ})G_{\circ}^{*}G_{\lambda}\xi_{u})\,.
$$ 
Therefore, by \eqref{1st}, 
$$
u_{\lambda}-u_{\mu}=G_{\mu}\xi_{u}(\mu)-G_{\lambda}\xi_{u}(\lambda)=
G_{\lambda}(\xi_{u}(\mu)-\xi_{u}(\lambda))+(\lambda-\mu)R^{\circ}_{\lambda}G_{\mu}\xi_{u}(\mu)\,.
$$
By \eqref{cap}, since $G_{\lambda}$ is injective, this gives $\xi_{u}(\mu)=\xi_{u}(\lambda)\equiv\xi_{u}$. 
Thus
$$
u=u_{\circ}+G_\circ\xi_{u}\,,
$$
where
$$u_{\circ}=u_{\lambda}+
(G_{\lambda}-G_{\circ})\xi_{u}
$$
and
$$
(\xi_{u},\tau u_{\circ})=(\xi_{u},\tau u_{\lambda}+(\lambda_{\circ}-\lambda)G_\circ^{*}G_{\lambda}\xi_{u})\in
\Theta\,.
$$
Then, by \eqref{2st}
\begin{align*}
A_{\Theta}(u)=A_{\circ}u_{\lambda}-\lambda G_{\lambda}\xi_{u}
=
A_{\circ}u_{\circ}+A_{\circ}(G_{\circ}-G_{\lambda})\xi_{u}-\lambda G_{\lambda}\xi_{u}
=
A_{\circ}u_{\circ}-\lambda_{\circ} G_{\circ}\xi_{u}\,.
\end{align*}
\end{proof}
\begin{remark}\label{char} By the characterization of $S^*$ given in \cite[Theorem 3.1]{[P04]}  one has 
$$
\D(S^{*})=\{u\in\H:u=
u_\circ+G_\circ\xi\,,\ u_\circ\in \D(A_{\circ})\,,\ \xi\in\fh\}\,,
$$
$$
S^{*}u=A_{\circ}u_\circ-\lambda_{\circ}G_{\circ}\xi
$$
and so
$$A_{\Theta}\subset S^*\,.
$$
Moreover
$$
S\subset A_{\Theta}\iff \N\subseteq\D(A_{\Theta}) \iff   
(0,0)\in\Theta\,\Longrightarrow\, \overline{\D(A_{\Theta})}=\H\,.
$$ 
Hence if $(0,0)\in\Theta$ then $A_{\Theta}$ generates a one-parameter continuous nonlinear semigroup defined on the whole Hilbert space $\H$.  \par
Since 
$$
\D(A_{\circ})\cap\D(A_{\Theta})=\{u\in\D(A_{\circ}):(0,\tau u)\in\Theta\}\,,
$$
one has
$$
\D(A_{\circ})\cap\D(A_{\Theta})\not=\emptyset\iff 0\in\D(\Theta)
$$
and 
$$
\forall u\in \D(A_{\circ})\cap\D(A_{\Theta})\,,\quad A_{\Theta}(u)=A_{\circ}u\,.
$$
Moreover, since $\{u\in\D(A_{\circ}):(0,\tau u)\in\Theta\}$ is closed and convex and $\tau$ is linear continuous, the set  $\D(A_{\circ})\cap\D(A_{\Theta})$ is convex and closed in $\H_{\circ}$.
\end{remark}
\begin{remark} If $\Theta^{-1}$ is single-valued then
$$
(A_{\Theta}+{\lambda_{\circ}})^{-1}=(A_{\circ}+{\lambda_{\circ}})^{-1}+G_{\circ}\Theta^{-1}
\!\circ G_{\circ}^{*}\,.
$$
\end{remark}
\begin{remark} Let $S$ be strictly positive (i.e $\omega<0$) and take $\lambda_{\circ}\in(\omega,0)$.  Further suppose that there exists $\xi$ such that $(\xi,\lambda_{\circ} G_{\circ}^{*}G_{\circ}\xi)\in\Theta$. Then the equation $A_{\Theta}u=0$ 
has the (necessarily unique) solution $u_{\infty}:=(\lambda_{\circ} A^{-1}_{\circ}+1)G_{\circ}\xi$ (notice that $u_{\infty}=0$ whenever $(0,0)\in\Theta$). Then, by \cite[Th\'eor\`eme 3.9]{[Br2]}, one obtains 
$$
\forall u\in\overline{\D(A_{\Theta})}\,,\quad \lim_{t\to+\infty}S^{\Theta}_t(u)=u_{\infty}\,, 
$$
more precisely
$$
\forall u\in\overline{\D(A_{\Theta})}\,,\quad 
\|S^{\Theta}_t(u) -u_{\infty}\|\le e^{\lambda_{\circ} t}\,\|u -u_\infty\|\,, 
$$
$$
\forall u\in\D(A_{\Theta})\,,\quad \left\|\frac{d^{+}}{dt}\, S^{\Theta}_{t}(u)\right\|\le e^{\lambda_{\circ}t}\|A_{\Theta}(u)\|\,,
$$
where $S_t^{\Theta}$ denotes the nonlinear semigroup of contractions generated by $A_{\Theta}$.
\end{remark}
Before stating the following convergence result, we recall the following definition: given the sequence 
$\{\Theta_{n}\}_{1}^{\infty}$, $\Theta_{n}\subset\fh\times\fh$, the relation $\liminf \Theta_{n}\subset\fh\times\fh$ is defined as the set of all couples $(\xi,\t\xi)\in\fh\times\fh$ such that there are sequences $\{\xi_{n}\}_{1}^{\infty}$, 
$\{\t\xi_{n}\}_{1}^{\infty}$, with $(\xi_{n},\t\xi_{n})\in\Theta_{n}$, $(\xi_{n},\t\xi_{n})\to (\xi,\t\xi)$ as $n\uparrow\infty$. 
\begin{lemma}\label{conv} Given the semi-bounded self-adjoint operator $A_{\circ}\ge-\omega$ and  the bounded surjective operators $\tau_n:\H_{\circ}\to\fh$ and  $\tau:\H_{\circ}\to\fh$, with kernels $\N_{n}$ and $\N$ dense in $\H$, define the symmetric operators $S_n:=A_{\circ}|\N_{n}$ and $S:=A_{\circ}|\N$. Given the maximal monotone relations $\Theta_n\subset \fh\times\fh$, $\Theta\subset \fh\times\fh$, let $A_{\Theta_n}$ and $A_{\Theta}$ be the maximal monotone operators of type $\lambda_{\circ}>\omega$ provided by Theorem \ref{estensioni}. If $\tau_{n}$ strongly converges to $\tau$ and $\Theta\subset \liminf \Theta_{n}$ then 
\begin{equation*}
\forall T\ge 0\,,\quad \forall u\in\overline{\D(A_{\Theta})}\,,\quad \lim_{n\to +\infty}\, \sup_{0\le t\le T}\|S^{\Theta_{n}}_{t}(u_{n})-S^{\Theta}_t(u)\|=0\,,
\end{equation*}
where $S^{\Theta_{n}}_{t}$ denotes the semi-group generated by $A_{\Theta_{n}}$, $u_{n}\in\overline{\D(A_{\Theta_{n}})}$ and $\|u_{n}-u\|\to 0$.
\end{lemma}
\begin{proof}
By our hypothesis on $\tau_n$,
$G_{n,\lambda}:=(\tau_n (A_{\circ}+\lambda)^{-1})^{*}$ and $G_{n,\lambda}^{*}$ 
strongly converge to
$G_\lambda$ and $G_\lambda^{*}$ respectively. 
This implies that $(\lambda-\lambda_{\circ})G_{n,\circ}^{*}G_{\lambda}$ strongly converges to 
$(\lambda-\lambda_{\circ})G_{\circ}^{*}G_{\lambda}$ and hence 
$M^{\Theta}_{\lambda}\subset\liminf(\Theta_{n}+(\lambda-\lambda_{\circ})G^{*}_{n,\circ}G_{\lambda})$. Therefore (see e.g. \cite[Proposition 4.4]{[Barb]}) 
$$
\forall\xi\in\fh\,,\qquad
\lim_{n\uparrow\infty}(\Theta_{n}+(\lambda-\lambda_{\circ})G_{n,\circ}^{*}G_{\lambda})^{-1}(\xi)=(M^{\Theta}_{\lambda})^{-1}(\xi)\,.
$$ 
The thesis then follows by
the resolvent formula provided in Theorem \ref{estensioni} and by the nonlinear Trotter-Kato Theorem (see e.g. \cite[Th\'eor\`eme 3.16]{[Br2]}). 
\end{proof}
\begin{remark}\label{complex} Let $A:\D(A)\subseteq \H_{\C}\to\H_{\C}$, where  
$\H_{\C}$ is a complex Hilbert space. Writing $\H_{\C}=\H_{\RE}+i\,\H_{\RE}$, where $\H_{\RE}$ is the realification of $\H_{\C}$ and defining $A_{1}$ and $A_{2}$ by the relation $
A(u_{1}+iu_{2})=A_{1}(u_{1},u_{2})+iA_{2}(u_{1},u_{2})$, the nonlinear  operator $A$ is said to be (maximal) monotone whenever its realification $A_{\RE}(u_{1}\oplus u_{2}):=A_{1}(u_1,u_{2})
\oplus A_{2}(u_{1},u_{2})$ is (maximal) monotone in the real Hilbert space $\H_{\RE}\oplus \H_{\RE}$. Thus the whole theory of maximal monotone operators in real Hilbert spaces extends, with the obvious modifications, to complex spaces and one can readily extend Theorem \ref{estensioni} to complex Hilbert spaces.  Moreover, since any skew-adjoint linear operator $W:\D(W)\subseteq
\H_{\C}\to\H_{\C}$ is maximal monotone, one also obtains a version of  Theorem \ref{estensioni} (having the same proof and statement; it suffices to replace $S$ with $W|\N$ and $A_{\circ}$ with $W$) providing maximal monotone extensions of skew-symmetric operators in complex Hilbert spaces (linear skew-adjoint extensions of skew-symmetric operators corresponding to abstract wave equations have been studied in \cite{[P05]}). By Remark \ref{char}, such extensions are maximal monotone restrictions of the linear operator $(W|\N)^{*}$ and so, in the case of skew-adjoint operators  in the complex Hilbert space $H_{0}\oplus H_{1}$ of the kind $W(u_{0}\oplus u_{1})=Du_{1}\oplus Gu_{0}$  (here $D=-G^{*}$, $D:\D(D)\subseteq H_{1}\to H_{0}$, $G:\D(G)\subseteq H_{0}\to H_{1}$),  one recovers the maximal monotones operators characterized in the recent paper \cite{[Tro]} (the author got knowledge of \cite{[Tro]} after the completion of this paper; he thanks Sascha Trostorff for the communication). 
\end{remark}
\section{Sub-potential extensions}
Let $\varphi:\fh\to(-\infty,+\infty]$ be a proper (i.e. not identically $+\infty$) convex function and let us define its (not empty) effective domain by $$\D(\varphi):=\{\xi\in\fh:\varphi(\xi)<+\infty\}\,;
$$
its {\it sub-differential} $\partial\varphi\subset \fh\times\fh$ is then defined by 
\begin{align*}
&\partial\varphi:=\{(\xi,\t\xi)\in\fh\times\fh:\forall \zeta\in\fh\,,\ \varphi(\xi) \le  \varphi(\zeta)+[ \t \xi,\xi-\zeta]\,\}
\\
\equiv&\{(\xi,\t\xi)\in\D(\varphi)\times\fh:\forall \zeta\in\D(\varphi)\,,\ \varphi(\xi)- \varphi(\zeta) \le [ \t \xi,\xi-\zeta]\,\}
\,.
\end{align*}
(here $[\cdot,\cdot]$ denotes the scalar product in the Hilbert space $\fh$).
Notice that $(\xi,0)\in\partial\varphi$ if and only if $\xi$ is a minimum point of $\varphi$. Also
notice that if 
$\varphi$ is G\^ ateaux-differentiable at $\xi$ then $\partial\varphi(\xi)=\nabla\varphi(\xi)$; so if $\varphi$ is everywhere G\^ ateaux-differentiable then $\partial\varphi=\nabla\varphi$. Sub-differentials of lower semi-continuous functions provide examples of maximal monotone operators (see e.g. \cite[Example 2.3.4., Proposition 2.12]{[Br2]}, \cite[Proposition 1.6]{[Barb]}): if $\varphi$ is lower semi-continuous then  $\partial\varphi$ is maximal monotone and $\text{\rm int}(\D(\partial\varphi))=\text{\rm int}(\D(\varphi))$, $\overline{\D(\partial\varphi)}=\overline{\D(\varphi)}$.
\par
An operator $\Theta=\partial\varphi$, $\varphi$ a proper and convex function, is called a {\it sub-potential monotone operator\,}; if $\Theta$ is maximal (this holds whenever $\varphi$ is lower semi-continuous) then we say that it is a {\it sub-potential maximal monotone operator}. 
\begin{remark}\label{sub} Let $\varphi$ be proper convex and let $\bar\varphi$
be its lower semi-continuous regularization, i.e. $\bar\varphi$ is the largest lower semi-continuous minorant of $\varphi$: $\text{epi$(\bar\varphi)$}=\overline{\text{epi$(\varphi)$}}$,
where the epigraph is defined by $\text{epi$(f)$}:=\{(\xi,\lambda)\in\fh\times \RE: f(\xi)\le \lambda\}$.
Then $\partial\varphi\subseteq\partial\bar\varphi$ and so $\partial\varphi=\partial\bar\varphi$ whenever $\partial\varphi$ is maximal monotone.
\end{remark}
Suppose that $L:\D(L)\subseteq \fh\to\fh$ is a non negative linear self-adjoint operator, so that it is maximal monotone. Then (see e.g. \cite[Proposition 2.15]{[Br2]}) $L=\partial\varphi_{L}$, where $\varphi_{L}:\fh\to [0,+\infty]$ is the proper  lower semi-continuous  convex function 
$$
\varphi_{L}:\fh\to[0,+\infty]\,,\quad \varphi_{L}(\xi):=\begin{cases}
\frac12\,| L ^{\frac12}\,\xi|^{2}\,,&\xi\in\D( L^{\frac12})\\
+\infty\,,&\text{otherwise}\,.
\end{cases}
$$
Hence one gets that $\xi\in\D( L^{\frac12})$ belongs to $\D(L)$ if and only if there exists $\tilde\xi\in\fh$ 
such that $\frac12\,| L^{\frac12} \,\xi|^{2}-\frac12\,| L^{\frac12} \,\zeta|^{2}\le[\tilde\xi,\xi-\zeta]$ for all $\zeta\in\D( L^{\frac12})$. In this case $L\xi=\tilde\xi$.
\par
Suppose that in Theorem \ref{estensioni} one has $\Theta=L$, $L$ a non negative 
linear self-adjoint operator and range$(G_{\circ})\cap\D((A_{\circ}+\lambda_{\circ})^{\frac12})=\{0\}$.
Then, by Theorem 2.4 in \cite{[P12]}, $A_{\Theta}+\lambda_{\circ}=\partial\Phi_{\circ}$, where the proper  convex function $\Phi_{\circ}:\H\to(-\infty,+\infty]$ is defined by
$$
\Phi_{\circ}(u):=\begin{cases}
\frac12\,\|(A_{\circ}+\lambda_{\circ})^{\frac12}u_{\circ}\|^{2}+\frac12\,|L^{\frac12} \,\xi|^{2}\,,&u\in\D(\Phi_{\circ})\\
+\infty\,,&\text{otherwise}\,,
\end{cases}
$$
$$
\D(\Phi_{\circ}):=\{u\in\H:u=u_{\circ}+G_{\circ}\xi\,,\ u_{\circ}\in\D((A_{\circ}+\lambda_{\circ})^{\frac12})\,,\ \xi\in\D(L^{\frac12})\}\,.
$$
Notice that the hypothesis range$(G_{\circ})\cap\D((A_{\circ}+\lambda_{\circ})^{\frac12})=\{0\}$ is needed in order that 
$\D(\Phi_{\circ})$ is well-defined. Also notice that $\Phi_{\circ}$ is lower semi-continuous since 
$\Phi_{\circ}(u)=\frac12\,\|(A_{\circ}+\lambda_{\circ})^{\frac12}u\|^{2}$ for any $u\in \D((A_{\circ}+\lambda_{\circ})^{\frac12})\equiv\D(\Phi_{\circ})$.
\par
A similar result holds in the nonlinear case:
\begin{theorem}\label{cyclic}
Let $\Theta=\partial\varphi\subset\fh\times\fh$ be a sub-potential maximal monotone operator and let $A_{\Theta}$ be defined as in Theorem \ref{estensioni}. Suppose that $\text{\rm range}(G_{\circ})\cap\D((A_{\circ}+\lambda_{\circ})^{\frac12})=\{0\}$ and define the proper convex function 
\begin{equation}\label{PHI}
\Phi:\H\to(-\infty,+\infty]\,,\ 
\Phi(u):=\begin{cases}\frac12\,\|(A_{\circ}+\lambda_{\circ})^{\frac12}u_{\circ}\|^{2}
+\varphi(\xi)&u\in\D(\Phi)\\
+\infty&\text{otherwise},
\end{cases}
\end{equation}
where
$$
\D(\Phi):=\{u\in\H:u=u_{\circ}+G_{\circ}\xi\,,\ u_{\circ}\in\D((A_{\circ}+\lambda_{\circ})^{\frac12})\,,\ \xi\in\D(\varphi)\}\,.
$$  
Then $A_{\Theta}+\lambda_{\circ}$ is a sub-potential maximal monotone operator: 
\begin{equation*}
A_{\Theta}+\lambda_{\circ}=\partial\Phi
\,.
\end{equation*}
\end{theorem} 
\begin{proof} Let us take $u=u_{\circ}+G_{\circ}\xi\in\D(A_{\Theta})$ and $v=v_{\circ}+G_{\circ}\zeta\in\D(\Phi)$. Then, by the definition of $A_{\Theta}$ and since $(\xi,\tau u_{\circ})\in\Theta=\partial\varphi$, one gets
\begin{align*}
&\langle (A_{\Theta}+\lambda_{\circ})(u), u-v\rangle
=\langle (A_{\circ}+\lambda_{\circ})u_{\circ}, u-v\rangle
\\=&
\langle (A_{\circ}+\lambda_{\circ})u_{\circ}, u_{\circ}-v_{\circ}\rangle+
\langle (A_{\circ}+\lambda_{\circ})u_{\circ}, G_{\circ}(\xi-\zeta)\rangle\\
=&\langle (A_{\circ}+\lambda_{\circ})u_{\circ}, u_{\circ}-v_{\circ}\rangle+
[G_{\circ}^{*}(A_{\circ}+\lambda_{\circ})u_{\circ}, \xi-\zeta]\\
=&\langle (A_{\circ}+\lambda_{\circ})u_{\circ}, u_{\circ}-v_{\circ}\rangle+
[\tau u_{\circ },\xi-\zeta]
\ge
\langle (A_{\circ}+\lambda_{\circ})u_{\circ}, u_{\circ}-v_{\circ}\rangle+\varphi(\xi)-\varphi(\zeta)\\
=&\frac12\,\|(A_{\circ}+\lambda_{\circ})^{\frac12}u_{\circ}\|^{2}+\varphi(\xi)
-\left(\frac12\,\|(A_{\circ}+\lambda_{\circ})^{\frac12}v_{\circ}\|^{2}+\varphi(\zeta)\right)
+
\frac12\,\|(A_{\circ}+\lambda_{\circ})^{\frac12}(u_{\circ}-v_{\circ})\|^{2}\\
\ge&\Phi(u)-\Phi(v)\,.
\end{align*}
Therefore graph$(A_{\Theta}+\lambda_{\circ})\subseteq\partial\Phi$. 
Let us now take $(u,\tilde u)\in\partial\Phi$, where $u=u_{\circ}+G_{\circ}\xi$ with $u_{\circ}\in\D((A_{\circ}+\lambda_{\circ})^{\frac12})$ and $\xi\in\D(\varphi)$.   Then, for any $v_{\circ}\in\D((A_{\circ}+\lambda_{\circ})^{\frac12})$ 
and setting $v:=v_{\circ}+G_{\circ}\xi$, one gets
\begin{align*}
&\langle\tilde u,u_{\circ}-v_{\circ}\rangle=\langle\tilde u,u-v\rangle\ge\Phi(u)-\Phi(v)
=
\frac12\,\|(A_{\circ}+{\lambda_{\circ}})^{\frac12}\,u_{\circ}\|^{2}
-\frac12\,\|(A_{\circ}+{\lambda_{\circ}})^{\frac12}\,v_{\circ}\|^{2}\,.
\end{align*}
Thus $u_{\circ}\in\D(A_{\circ})$ and $\tilde u=(A_{\circ}+\lambda_{\circ})u_{\circ}$. Next, for any $\zeta\in\D(\varphi)$, now setting $v:=u_{\circ}+G_{\circ}\zeta$, one gets
\begin{align*}
&[G_{\circ}^{*}\tilde u,\xi-\zeta]=\langle\tilde u,G_{\circ}\xi-G_{\circ}\zeta\rangle
\ge\Phi(u)-\Phi(v)=\varphi(\xi)-\varphi(\zeta)\,.
\end{align*}
Thus $(\xi,G_{\circ}^{*}\tilde u)\in\partial\varphi=\Theta$. Since $G_{\circ}^{*}\tilde u=\tau R^{\circ}_{\lambda_{\circ}}\tilde u=\tau u_{\circ}$, one obtains $(u,\tilde u)\in$graph$(A_{\Theta}+\lambda_{\circ})$ and so in conclusion graph$(A_{\Theta}+\lambda_{\circ})=\partial\Phi$. 
\end{proof}
\begin{remark}\label{Remark 4.7} By Remark \ref{sub}, in Theorem \ref{cyclic} one has 
$A_{\Theta}+\lambda_{\circ}=\partial\bar\Phi$, 
where   $\bar\Phi$ denotes the lower semi-continuous regularization of $\Phi$. Hence, in case $\lambda_{\circ}=0$ and range$(G_{\circ})\cap \D(A_{0}^{1/2})=\{0\}$, for any lower 
semi-continuous proper convex function $\varphi:\fh\to(-\infty,+\infty]$, one can define a maximal monotone operator $A_{\varphi}\subset S^{*}$ by $A_{\varphi}:=\partial\bar\Phi$, where $\Phi$ is given by \eqref{PHI}. Such a sub-differential $\partial\bar\Phi$ is fully described by Theorem \ref{estensioni}, since it coincides with the operator $A_{\partial\varphi}$.
\end{remark}
\begin{remark}\label{Remark 4.5} By the properties  of semigroups generated by sub-potential maximal monotone operators (see \cite[Chapter III, Section 3]{[Br2]} and \cite[Theorem 4.11, Corollary 4.4 and Remark 4.5]{[Barb]}), one gets the following regularity results about the nonlinear semigroup $S^{\varphi}_{t}$ generated by the nonlinear operator 
$A_\varphi:= \partial\Phi-\lambda_{\circ }=\partial\bar\Phi-\lambda_{\circ }
$ provided in Theorem \ref{cyclic}:
$$
\forall u\in\overline{\D(A_\varphi)}\,,\ \forall t>0\,,\quad S^{\varphi}_t(u)\in \D(A_\varphi)\,,
$$ 
$$
\forall u\in\overline{\D(A_\varphi)}\,,\ \forall v\in\D(A_\varphi)\,,\ \forall t>0\,,\quad
\left\|\frac{d^{+}}{dt}\, S^{\varphi}_t(u)\right\|\le\|A_\varphi v\|+\frac1t\,\|u-v\|\,,
$$
$$
\forall u\in\overline{\D(A_\varphi)}\,,\ \forall T>0\,,\quad \int_{0}^{T}t\,\left\|\frac{d\,}{dt}\, S^{\varphi}_t(u)\right\|^{2}dt<+\infty\,,
$$
$$
\forall u\in{\D(\bar\Phi)}\,,\ \forall T>0\,,\quad \int_{0}^{T}\left\|\frac{d\,}{dt}\, S^{\varphi}_t(u)\right\|^{2}dt<+\infty\,,
$$
$$
\forall u\in\overline{\D(A_\varphi)}\,,\ \forall T>0\,,\quad \int_{0}^{T}|\bar\Phi(S^{\varphi}_t(u))|\,dt<+\infty\,,
$$
$$
\forall u\in{\D(\bar\Phi)}\,,\ \forall T>0\,,\quad \int_{0}^{T}\left|\frac{d\,}{dt}\, \bar\Phi(S^{\varphi}_t(u))\right|\,dt<+\infty\,.
$$
\end{remark}
\begin{remark}\label{Remark 4.6}  Suppose $S>0$ and take $\lambda_{\circ}=0$. Then
\begin{align*}
&\text{kernel}(A_\varphi)=\{u\in\H: u=G_{\circ}\xi\,,\ (\xi,0)\in\partial\varphi\}
\\
\equiv
&
\{u\in\H: u=G_{\circ}\xi\,,\ \text{$\xi$ a minimum point of $\varphi$}\}\,.
\end{align*}
In case $\text{kernel}(A_\varphi)\not=\emptyset$, by \cite[Theorem 4]{[Bru]} one has 
\begin{equation}\label{w}
\forall u\in\overline{\D(A_\varphi)}\,, \ \exists u_{\infty}\in\text{kernel}(A_\varphi)\,:\,\text{w-}\lim_{t\to+\infty}S^{\varphi}_t(u)=u_{\infty}\,.
\end{equation}
By \cite[Theorem 5]{[Bru]}, if $\varphi$ is an even function then the above weak limit \eqref{w} becomes a strong one.
\end{remark}
\begin{remark}\label{moreau}
Let $\varphi_{\lambda}$ be the Moreau regularization of $\varphi$, i.e.
$$
\varphi_{\lambda}:\fh\to\RE\,,\quad\varphi_{\lambda}(\xi):=
\inf\left\{\frac{|\xi-\zeta|^{2}}{2\lambda}+\varphi(\zeta)\,;\ \zeta\in\fh\right\}\,.
$$
Then $\varphi_{\lambda}$ is convex and Fr\'echet differentiable on $\fh$ with $\nabla\varphi_{\lambda}=(\partial\varphi)_{\lambda}$, where $(\partial\varphi)_{\lambda}$ denotes the Yosida approximation of $\partial\varphi$, i.e. $(\partial\varphi)_{\lambda}=\frac1\lambda\,(1-(\lambda \partial\varphi+1)^{-1})$  (see e.g. \cite[Theorem 2.9]{[Barb]}). Let $\Phi_{\lambda}$ be the proper convex function defined as in \eqref{PHI} with $\varphi$ replaced by $\varphi_{\lambda}$. Then, denoting by $S^{\varphi_{\lambda}}_{t}$ the nonlinear semigroup generated by $A_{\varphi_{\lambda}}:=\partial\bar\Phi_{\lambda}$ (here we suppose $\lambda_{\circ}=0$), by Lemma \ref{conv} one has 
\begin{equation*}
\forall T\ge 0\,,\quad \forall u\in\overline{\D(A_{\varphi})}\,,\quad \lim_{\lambda\to 0}\, \sup_{0\le t\le T}\|S^{\varphi_{\lambda}}_{t}(u_{\lambda})-S^{\varphi}_t(u)\|=0\,,
\end{equation*}
where $u_{\lambda}\in\overline{\D(A_{\varphi_{\lambda}})}$ and $\|u_{\lambda}-u\|\to 0$.
\end{remark}
\section{Applications.}
\begin{subsection}{Nonlinear point perturbations of the Laplacian}\label{E1} Let $$A_{\circ}:H^2(\RE^3)\subseteq L^2(\RE^3)\to L^2(\RE^3)\,,\qquad
A_{\circ}u=-\Delta u\,,$$ 
$$\tau :H^2(\RE^3)\to\RE^n\,,\qquad
\tau u\equiv\{u(y)\}_{y\in Y}\,,$$
where $Y\subset\RE^3$ is a discrete set with $n$ elements. 
Here $H^2(\RE^3)\subset C_b(\RE^3)$ 
denotes the usual Sobolev-Hilbert space of square
integrable functions with square integrable second order
(distributional) partial derivatives. Thus we are looking for nonlinear maximal monotone extensions of the positive symmetric operator
$$
S:\D(S)\subseteq L^2(\RE^3)\to L^2(\RE^3)\,, \quad S u=-\Delta u\,,
$$
$$
\D(S):=\left\{u\in H^2(\RE^3)\,:\, u(y)=0\,,\ y\in Y \right\}\,.
$$
Since the kernel of the resolvent of $-\Delta$ is given by 
$$(-\Delta+\lambda)^{-1}(x_{1},x_{2})
=\frac{e^{-\sqrt \lambda\,|x_{1}-x_{2}|}}{4\pi|x_{1}-x_{2}|}\,,\qquad \lambda>0\,,$$ 
one has
\begin{equation*}
G_{\lambda}:\RE^n\to L^2(\RE^3)\,,\qquad [G_{\lambda}\xi](x)=\sum_{y\in Y}
\frac{e^{-\sqrt \lambda\,|x-y|}}{4\pi|x-y|}\
\xi_y\,,\quad \xi\equiv\{\xi_{y}\}_{y\in Y}\,,
\end{equation*}
and
\begin{equation*}
G_{\lambda}^*:L^2(\RE^3)\to\RE^n\,,\quad G_{\lambda}^*u\equiv\{(G_{\lambda}^*u)_y\}_{y\in Y}\,, 
\end{equation*}
$$
(G_{\lambda}^*u)_y :=\int_{\RE^3}\frac{e^{-\sqrt \lambda\,|x-y|}}{4\pi|x-y|}\,
u(x)\,dx\,.
$$
Taking $\lambda_{\circ}>\omega=0$, one gets
\begin{align*}
&(M^{\circ}_{\lambda}\xi)_{y}=(\lambda-\lambda_{\circ})(G_{\circ}^*G_{\lambda}\xi)_y=(\tau(G_{\circ}-G_{\lambda})\xi)_y\\
=&\lim_{x\to y}\frac{e^{-\sqrt\lambda_{\circ}\,|x-y|}-e^{-\sqrt\lambda\,|x-y|}}{4\pi|x-y|}\ \xi_y
+\sum_{y'\not=y}\frac{e^{-\sqrt\lambda_{\circ}\,|y-y'|}-e^{-\sqrt \lambda\,|y-y'|}}{4\pi|y-y'|}\,
\xi_{y'}\,
\end{align*}
so that $M^{\circ}_{\lambda}:\RE^n\to\RE^n$, $\lambda>0$, is represented by a
matrix with components 
\begin{equation*}
(M^{\circ}_{\lambda})_{yy'}
=
\begin{cases}\frac{\sqrt\lambda-\sqrt\lambda_{\circ}}{4\pi}&y=y'\\
\frac{e^{-\sqrt\lambda_{\circ}\,|y-y'|}-e^{-\sqrt \lambda\,
|y-y'|}}{4\pi|y-y'|}& y\not=y'\,.\end{cases}
\end{equation*}
For any nonlinear maximal monotone relation $\tilde\Theta\subset\RE^{n}\times\RE^{n}$ by Theorem \ref{estensioni} we get the maximal monotone operator of type $\lambda_{\circ}>0$
$$
(-\Delta)_{\tilde\Theta}:\D(-\Delta_{\tilde\Theta})\subseteq L^{2}(\RE^{3})\to L^{2}(\RE^{3})\,,\quad 
(-\Delta)_{\tilde\Theta}u=-\Delta u_{\circ}-\lambda_{\circ}G_{\circ}\xi_{u}\,,
$$
$$
\D((-\Delta)_{\tilde\Theta})=\{u\in L^{2}(\RE^{3}): u=u_{\circ}+G_{\circ}\xi_{u}\,,\ u_{\circ}\in H^{2}(\RE^{3})\,,\ (\xi_{u},\tau u_{\circ})\in\tilde\Theta\}\,.
$$
Now we give an alternative representation of the nonlinear extensions which are more tied to the linear ones presented in the book \cite{[AGHH]} and which generate, under suitable conditions on the extension parameter $\Theta$, contraction nonlinear semigroups. \par
We define 
$$
G_{0}:\RE^{n}\to L^{2}_{loc}(\RE^{3})\,,\quad [G_{0}\xi](x)=\frac1{4\pi}\sum_{y\in Y}
\frac{\xi_{y}}{|x-y|}\,,
$$
Then $(G_{\circ}
-G_{0})\xi$ belongs to  
$$\tilde H^{2}(\RE^{3})=\{u\in C_{b}(\RE^{3}):\|\nabla u\|\in L^{2}(\RE^{3})\,,\ \Delta u\in L^{2}(\RE^{3})\}\,,$$ and
$$
\Delta (G_{\circ}-G_{0})\xi=\lambda_{\circ}G_{\circ}\xi\,,
$$
so that  
$$
(-\Delta)_{\tilde\Theta}u=-\Delta u_{0}\,,\quad u_{0}:=u_{\circ}+(G_{\circ}-G_{0})\xi\,.
$$
Given $u=u_{0}+G_{0}\xi$, $u_{0}\in \tilde H^{2}(\RE^{3})$, let us now define $\tilde\tau u\equiv\{(\tilde\tau u)_{y}\}_{y\in Y}\in \RE^{n}$  by 
$$
(\tilde\tau u)_{y}:=\lim_{x\to y}\left(u(x)-\frac{1}{4\pi}\,
\frac{\xi_{y}}{|x-y|}\right)\,,
$$
so that 
$$
\tau u_{\circ}=\tilde \tau u_{0}+L^{\circ}\xi\,,
$$
where the symmetric linear operator $L^{\circ}:\RE^{n}\to\RE^{n}$ is represented by a matrix with components
$$
L^{\circ}_{yy'}=
\begin{cases}\frac{\sqrt\lambda_{\circ}}{4\pi}&y=y'\\
-\frac{e^{-\sqrt \lambda_{\circ}\,
|y-y'|}}{4\pi|y-y'|}& y\not=y'\,.\end{cases}
$$
Posing
$$
\Theta:=\tilde\Theta-L^{\circ}\,,\quad M_{\lambda}:=M^{\circ}_{\lambda}+L^{\circ}
$$
we can re-define the extensions by $(-\Delta)_{\tilde\Theta}\equiv(-\Delta)_{\Theta}$ and so one obtains the following result:
\begin{proposition} Let $M_{\lambda}:\RE^{n}\to\RE^{n}$, $\lambda\ge 0$, be represented by the matrix
$$
(M_{\lambda})_{yy'}=\frac1{4\pi}
\begin{cases}\sqrt\lambda&y=y'\\
-\frac{e^{-\sqrt\lambda\,|y-y'|}}{|y-y'|}& y\not=y'\,.\end{cases}
$$
Let $\Theta\subset \RE^{n}\times\RE^{n}$ be a nonlinear maximal monotone relation of type $\gamma_{0}$, where $\gamma_{0}$ is the smallest eigenvalue of $M_{0}$. 
Then 
$$
((-\Delta)_{\Theta}+\lambda)^{-1}=(-\Delta+{\lambda})^{-1}+G_{\lambda}(\Theta+M_{\lambda})^{-1}\circ G_{\lambda}^{*}\,,\quad \lambda>0\,.
$$
is the nonlinear resolvent of the nonlinear maximal monotone operator
$$
(-\Delta)_{\Theta}\,u:=-\Delta u_{0}\,,
$$
$$
\D((-\Delta)_{\Theta})=\{u\in L^{2}(\RE^{3}): u=u_{0}+G_{0}\xi_{u}\,,\ u_{0}\in \tilde H^{2}(\RE^{3})\,,\ (\xi_{u},\tilde\tau u)\in\Theta\}\,.
$$
\end{proposition}
\begin{proof} By the previous calculations we know that $(-\Delta)_{\Theta}$ is maximal monotone of type $\lambda_{\circ}>0$ and so we only need to show that $\langle(-\Delta)_{\Theta}(u)-(-\Delta)_{\Theta}(v),u-v\rangle\ge0$. One has
\begin{align*}
&\langle(-\Delta)_{\Theta}(u)-(-\Delta)_{\Theta}(v),u-v\rangle=-\langle\Delta(u_{0}-v_{0}),u-v\rangle\\
=&-\int_{\RE^{3}}\Delta (u_{0}-v_{0})(x)(u_{0}(x)-v_{0}(x))\,dx
-\sum_{y\in Y}\left(\int_{\RE^{3}}\frac{\Delta (u_{0}-v_{0})(x)}{4\pi\,|x-y|}\,dx\right)\,(\xi_{u}-\xi_{v})_{y}\\
= &\int_{\RE^{3}}\|\nabla (u_{0}-v_{0})\|^{2}(x)\,dx+\sum_{y\in Y}(u_{0}-v_{0})(y)(\xi_{u}-\xi_{v})_{y}\\
\ge&\sum_{y\in Y}\left((\tilde\tau(u-v))_{y}-\frac1{4\pi}\sum_{y'\in Y\backslash\{y\}}\frac{(\xi_{u}-\xi_{v})_{y'}}{|y-y'|}\right)(\xi_{u}-\xi_{v})_{y}\\
=&\sum_{y\in Y}((\tilde\xi_{u}-\tilde\xi_{v}+M_{0}(\xi_{u}-\xi_{v}))_{y}(\xi_{u}-\xi_{v})_{y}
\,,
\end{align*}
where $(\xi_{u},\tilde\xi_{u})$ and $(\xi_{v},\tilde\xi_{v})$ belong to $\Theta\subset\RE^{n}\times\RE^{n}$. Thus
the nonlinear extension $(-\Delta)_{\Theta}$ is maximal monotone whenever $\Theta$ is maximal monotone of type $\gamma_{0}$.\end{proof}
\begin{remark}
By the previous theorem, if $Y$ is a singleton then $(-\Delta)_{\Theta}$ generates a contraction nonlinear semigroup $S^{\Theta}_{t}$ for any maximal monotone relation $\Theta\subset\RE\times\RE$. 
Since 
$\Theta=\partial\varphi$ for some proper lower semicontinuous convex function $\varphi:\RE\to (-\infty,+\infty]$ (see \cite[Example 2.8.1]{[Br2]}), $S^{\Theta}_{t}(u)\in\D((-\Delta)_{\Theta})$ for any $t>0$ and for any $u\in \overline{\D((-\Delta)_{\Theta})}$ (for any $u\in L^{2}(\RE^{3})$ in case $(0,0)\in\Theta\,$). 
\end{remark}
\begin{remark} Since $-\Delta (G_{0}\xi)_{y}=\xi_{y}\,\delta_{y}$, where $\delta_{y}$ denotes the Dirac mass at $y$, one has $
(-\Delta)_{\Theta} (u)=-(\Delta u+\sum_{y\in Y}(\xi_{u})_{y}\,\delta_{y})$. In partucular, if $\Theta^{-1}$ is single-valued then, setting $\alpha_{y}(\zeta):=(\Theta^{-1}(\zeta))_{y}$,
one obtains $(-\Delta)_{\Theta} (u)=-(\Delta u+\sum_{y\in Y}\alpha_{y}(\tilde\tau u)\delta_{y})$.
\end{remark}
\end{subsection}
\begin{subsection}{The Laplacian with nonlinear boundary conditions on a bounded domain}\label{E3} Here we follow the same approach provided in \cite[Example 5.5]{[P08]} for the linear case, to which we refer for more details and related references.\par 
Given $\Omega\subset\RE^n$, $n>1$, 
a bounded open set with a boundary $\Gamma$ which is a
smooth embedded sub-manifold 
(these hypotheses could
be weakened), $H^m(\Omega)$ denotes the usual Sobolev-Hilbert
space of functions on $\Omega$ with square integrable partial
(distributional) 
derivatives of any order $k\le m$ and $H^s(\Gamma)$,
$s$ real, denotes the fractional order Sobolev-Hilbert space defined, 
since here $\Gamma$ can be made a smooth compact Riemannian
manifold, as the
completion of $C^\infty(\Gamma)$ with respect to the scalar
product 
$$
\langle f,g\rangle_{H^s(\Gamma)}
:=\langle f,(-\Delta_{LB}+1)^{s}g\rangle_
{L^2(\Gamma)}\,.
$$ 
Here the self-adjoint
operator $\Delta_{LB}$ is the Laplace-Beltrami operator in 
$L^2(\Gamma)$. With such a definition $(-\Delta_{LB}+1)^{s/2}$ 
can be extended to a unitary map, 
which we denote by the same
symbol, 
$$(-\Delta_{LB}+1)^{s/2}:H^{r}(\Gamma)
\to H^{r-s}(\Gamma)\,.
$$
For successive notational convenience we pose
$$
\Lambda:=(-\Delta_{LB}+1)^{1/2}:H^{s}(\Gamma)
\to H^{s-1}(\Gamma)\,,\quad \Sigma:=\Lambda^{-1}\,.
$$
The continuous and surjective linear operator 
$$
\gamma:H^2(\Omega)\to
H^{3/2}(\Gamma)\times
H^{1/2}(\Gamma)\,,\qquad \gamma u:=(\gamma_{0}u,\gamma_{1}u)\,,
$$
is defined as the unique bounded linear operator such that, in the case $u\in C^{\infty}(\bar\Omega)$,  
$$
\gamma_{0}u\,(x):=u\,(x)\,,\quad
\gamma_{1} u\,(x):=n(x){}^{}\!\cdot\!\nabla u\,(x)
\equiv\frac{\partial
  u}{\partial n}\,(x)\,,\quad x\in\Gamma\,.
$$
Here $n$ denotes the inner normal vector on $\Gamma$. The map $\gamma$  can be further extended to a bounded liner operator
$$
\hat\gamma:\D(\Delta^{\!\max})\to H^{-1/2}(\Gamma)\times
H^{-3/2}(\Gamma)\,,\qquad \hat \gamma u
=(\hat\gamma_{0}u,\hat\gamma_{1}u)\,,
$$
where
$$
\D(\Delta^{\!\max}):=\left\{u\in L^2(\Omega)\,:\, \Delta u\in L^2(\Omega)\right\}\,.
$$ 
Now let $A_{\circ}=-\Delta_{D}$ be the self-adjoint operator in $L^2(\Omega)$ given by the
Dirichlet Laplacian
$$
\Delta_{D}: \D(\Delta_{D})\subseteq L^2(\Omega)\to L^2(\Omega)\,\quad 
\Delta_D u=\Delta u\,,
$$
\begin{align*}
\D(\Delta_{D}):=H^2(\Omega)\cap H^1_{0}(\Omega)\,,\quad H^1_{0}(\Omega):=\left\{u\in
H^1(\Omega)\,:\,\gamma_{0}u=0\right\}\,.
\end{align*}
We take $\fh=H^{1/2}(\Gamma)$ and $\tau=\gamma_{1}|\D(\Delta_{D})$. 
Thus we are looking for nonlinear maximal monotone extensions of the strictly positive 
symmetric operator $S=-\Delta^{\!\min}$ given by the minimal Laplacian 
$$\Delta^{\!\min}:\D(\Delta^{\!\min})\subseteq L^2(\Omega)\to
L^2(\Omega)\,,\qquad \Delta^{\!\min}u:=\Delta u\,,$$
\begin{align*}
\D(\Delta^{\!\min}):=\left\{u\in
H^2(\Omega)\,:\,\gamma_{0}u=\gamma_{1}u=0\right\}\,.
\end{align*}
Notice that by defining the maximal Laplacian
$\Delta^{\!\max}$ as the distributional Laplacian
restricted to $\D(\Delta^{\!\max})$,
one has $\Delta^{\!\max}=(\Delta^{\!\min})^*$.\par
Posing $R^{D}_{\lambda}:=(-\Delta_{D}+\lambda)^{-1}$, by \cite[Example 5.5]{[P08]} one has (here $\lambda>\lambda_{\circ}=0$)
$$
M^{\circ}_{\lambda}:H^{1/2}(\Gamma)\to H^{1/2}(\Gamma)\,,\quad 
M^{\circ}_{\lambda}= \lambda\,\gamma_{1}R^D_\lambda K\Lambda\,,
$$
$$
G_{\lambda}:H^{1/2}(\Gamma)\to L^{2}(\Omega)\,,\quad G_{\lambda}=-\Delta_D 
R^D_\lambda K\Lambda\,,
$$
where $K:H^{-1/2}(\Gamma)\to \D(\Delta^{\!\max})$ is the Poisson operator, i.e. $K$ is the continuous linear
operator which solves the  
Dirichlet boundary value problem
\begin{align*}
\begin{cases}
\Delta_{\max} Kf=0\,,& \\
\hat\gamma_{0}\, Kf=f\,.&
\end{cases}
\end{align*} 
Combining the reasonings in \cite[Example 5.5]{[P08]} with Theorem \ref{estensioni} 
one obtains, given any monotone relation $\Theta\subset H^{1/2}(\Gamma)\times H^{1/2}(\Gamma)$, the nonlinear maximal monotone operators $(-\Delta)_{\Theta}$ defined by
$$
(-\Delta)_{\Theta}:\D(-\Delta_{\Theta})\subseteq L^2(\Omega)\to
L^2(\Omega)\,,\qquad (-\Delta)_{\Theta}u=-\Delta u\,,
$$
\begin{align}\label{D}
\D((-\Delta)_{\Theta})
=\left\{u\in
\D(\Delta^{\!\max})\,:\,(\Sigma\hat\gamma_{0}u,\hat\gamma_{1}u-P_{0} \hat\gamma_{0}u)\in\Theta
\right\}
\end{align}
with nonlinear resolvent 
\begin{equation}\label{R}
((-\Delta)_{\Theta}+\lambda)^{-1}
=R^D_\lambda-\Delta_D 
R^D_\lambda K\Lambda(\Theta+\lambda\,\gamma_{1}R^D_\lambda K\Lambda)^{-1}\circ \gamma_{1}R^D_\lambda\,,\quad \lambda>0\,.
\end{equation}
Here 
$$
P_{0} :H^{s}(\Gamma)\to 
H^{s-1}(\Gamma)\,,\quad s\ge -\frac12\,,\quad
P_{0} :=\hat\gamma_{1}\, K\,,
$$ 
is the Dirichlet-to-Neumann operator. The relation $\hat\gamma_{1}-P_{0} \hat\gamma_{0}=\gamma_{1}\Delta_{D}^{-1}\Delta$ shows that $\hat\gamma_{1}u-P_{0} \hat\gamma_{0}u\in H^{1/2}(\Gamma)$ for any $u\in \D(\Delta^{\!\max})$, so that the nonlinear boundary condition appearing in $\D((-\Delta)_{\Theta})$ is well-defined. \par 
Let $\psi:H^{-1/2}(\Gamma)\to(-\infty,+\infty]$ be a proper convex function such that $\Theta=\partial(\psi\circ\Lambda)\subset H^{1/2}(\Gamma)\times H^{1/2}(\Gamma)$ is maximal monotone. Then, by $G_{0}=K\Lambda$ and $\hat\gamma_{0}Kf=f$, Theorem \ref{cyclic} gives $(-\Delta)_{\partial(\psi\circ\Lambda)}=\partial\Phi$, where
$$
\Phi(u):=\begin{cases}\frac12\,\|\nabla(u-K\hat\gamma_{0}u)\|^{2}+\psi(\hat\gamma_{0}u)&u\in\D(\Phi)\\
+\infty&\text{otherwise}\,,
\end{cases}
$$
$$
\D(\Phi):=\{u\in \D(\Delta^{\!\max}):u-K\hat\gamma_{0}u\in H^{1}(\Omega)\,,\ \hat\gamma_{0}u\in\D(\psi)\}\,.
$$
By elliptic regularity, if $u\in\D(\Delta^{\!\max})$ and $\hat\gamma_{0}u\in H^{1/2}(\Gamma)$ then both $u$ and $K\hat\gamma_{0}u$ belong to $H^{1}(\Omega)$. If we further suppose that $\D(\psi)\subseteq H^{1/2}(\Gamma)$ then  $\D(\Phi)\subseteq H^{1}(\Omega)$ and 
$$
\frac12\,\|\nabla u\|^{2}=\frac12\,\|\nabla (u-K\gamma_{0}u)\|^{2}+\frac12\,\|\nabla K\gamma_{0}u\|^{2}=\frac12\,\|\nabla (u-K\gamma_{0}u)\|^{2}+\phi_0 (\gamma_{0}u)\,,
$$
where the proper lower semicontinuous convex function $\phi_0 :L^{2}(\Gamma)\to [0,+\infty]$ is defined by
$$\phi_0 (f):=\begin{cases}-\frac12\,(P_{0}f,f)\,,& f\in H^{1/2}(\Gamma)\\
+\infty\,,&\text{otherwise}\,,
\end{cases}
$$
(here $(\cdot,\cdot)$ denotes the $H^{-1/2}(\Gamma)$-$H^{1/2}(\Gamma)$ duality). Thus one can re-define $\Phi$ by
\begin{equation}\label{PHI1}
\Phi(u):=\begin{cases}\frac12\,\|\nabla u\|^{2}+\varphi(\gamma_{0}u)\,,& u\in\D(\Phi)\\
+\infty\,,&\text{otherwise}\,,
\end{cases}
\end{equation}
$$
\D(\Phi)=\{u\in H^{1}(\Omega):\gamma_{0}u\in\D(\varphi)\}\,,
\qquad
\varphi:=\psi-\phi_0 \,,\quad \D(\varphi)=\D(\psi)\,.
$$
Therefore, setting $\psi:=\varphi+\phi_0 $, we can define a maximal monotone operator $(-\Delta)_{\varphi}\subset \Delta^{\!\max}$ for any proper convex function $\varphi:L^2(\Gamma)\to(-\infty,+\infty]$, $\D(\varphi)\cap H^{1/2}(\Gamma)\not=\emptyset$, such that $\partial((\varphi+\phi_0 )\circ\Lambda)\subset H^{1/2}(\Gamma)\times H^{1/2}(\Gamma)$ is maximal monotone. The next  result provides some sufficient conditions: 
\begin{lemma}\label{suff}
Let $\varphi:L^2(\Gamma)\to(-\infty,+\infty]$, $\D(\varphi)\cap H^{1/2}(\Gamma)\not=\emptyset$, be a proper lower semicontinuous convex function. Then $(\varphi+\phi_0 )\circ\Lambda$ is a  proper lower semicontinuous convex function in  $H^{1/2}(\Gamma)$. Hence 
$\partial((\varphi+\phi_0 )\circ\Lambda)\subset H^{1/2}(\Gamma)\times H^{1/2}(\Gamma)$ is maximal monotone. 
\end{lemma}
\begin{proof} Let $f_{n}\to f$ in $H^{1/2}(\Gamma)$;  without loss of generality we can suppose that the numerical 
sequence $(\varphi+\phi_0 )(\Lambda f_n)$ is bounded. Since the non negative self-adjoint operator $P_{0}:H^{1}(\Gamma)\subseteq L^{2}(\Gamma)\to L^{2}(\Gamma)$ has compact resolvent and ker$(P_{0})=\RE$, denoting by $\mu_{0}>0$ its first positive eigenvalue and setting $\langle f\rangle:=(\text{\rm vol}(\Gamma))^{-1/2}\int_{\Gamma} f(x)\,d\sigma(x)$, one has 
$\phi(\Lambda f_n)\ge \mu_{0}(\|\Lambda f_{n}\|_{L^{2}}^{2}-\langle\Lambda f_{n}\rangle^{2})$. Since $\varphi$ is bounded from below by an affine function (see \cite[Proposition 1.1]{[Barb]}) and $\langle\Lambda f_{n}\rangle^{2}=\langle f_{n}\rangle^{2}\le\|f_{n}\|^{2}_{L^{2}}$, in conclusion $\|\Lambda f_{n}\|_{L^{2}}$ is bounded. Thus, taking a subsequence, we have 
$\Lambda f_{n}\rightharpoonup\Lambda f$ in $L^{2}(\Gamma)$. This conclude the proof since both $\phi_{0}$ and $\varphi$ are lower semicontinuous and hence weakly lower semicontinuous.
\end{proof}
By the previous Lemma we obtain sub-potential maximal monotone realizations of the Laplacian with nonlinear Robin-type boundary conditions: 
\begin{proposition}\label{robin} Let $\varphi:L^2(\Gamma)\to(-\infty,+\infty]$ be a proper lower semicontinuous convex function such that 
\text{\rm int$(\D(\varphi))\cap H^{1/2}(\Gamma)\not=\emptyset$}. Then
$$
(-\Delta)_\varphi:\D((-\Delta)_\varphi)\subseteq L^{2}(\Omega)\to L^{2}(\Omega)\,,\quad
(-\Delta)_\varphi(u):=-\Delta u\,,
$$
$$
\D((-\Delta)_{\varphi})=\left\{u\in \D(\Delta^{\!\max})\cap H^{3/2}(\Omega): (\gamma_{0}u,\gamma_{1}u)\in \partial\varphi\right\}
$$ 
is sub-potential maximal monotone, 
$$
(-\Delta)_\varphi=\partial\Phi\,, \quad \Phi(u)=\frac12\,\|\nabla u\|^{2}+\varphi(\gamma_{0}u)\,,\quad \D(\Phi)=\{u\in H^{1}(\Omega):\gamma_{0}u\in\D(\varphi)\}\,.
$$
Its nonlinear resolvent is given by 
$$
((-\Delta)_{\varphi}+\lambda)^{-1}
=R^D_\lambda-\Delta_D 
R^D_\lambda K(\partial\varphi-P_{\lambda})^{-1}\circ \gamma_{1}R^D_\lambda\,,\quad \lambda>0\,,
$$
where the Dirichlet-to-Neumann operator $P_{\lambda}$ is defined by 
$$
P_{\lambda}:=P_{0}-\lambda\,\gamma_{1}R^D_\lambda K\,.
$$
\end{proposition}
\begin{proof} At first let us notice that given a proper convex function $\psi:L^{2}(\Gamma)\to(-\infty,+\infty]$ and considering the proper convex function   $\psi\circ\Lambda: H^{1/2}(\Gamma)\to(-\infty,+\infty]$ with domain $\D(\psi\circ\Lambda):=\{f\in H^{1}(\Gamma):\Lambda f\in \D(\psi)\}$, one gets $$\partial(\psi\circ\Lambda)=(\partial\psi\circ\Lambda)\cap (H^{1/2}(\Gamma)\times H^{1/2}(\Gamma))\,,
$$ where $(\partial\psi\circ\Lambda)(f ):=\partial\psi(\Lambda f)$. Thus, by  int$(\D(\varphi))\cap H^{1/2}(\Gamma)=\emptyset$,  and by \cite[Corollarie 2.11]{[Br2]}, one obtains 
\begin{align*}
\partial((\varphi+\phi_0 )\circ\Lambda)=&(\partial\varphi\circ\Lambda+\partial\phi_0 \circ\Lambda)\cap (H^{1/2}(\Gamma)\times H^{1/2}(\Gamma))\\
=&
(\partial\varphi\circ\Lambda-P_{0}\Lambda)\cap (H^{1/2}(\Gamma)\times H^{1/2}(\Gamma))\,.
\end{align*} 
Then, by \eqref{D}, one gets $\hat\gamma_{0}u\in\D(\partial\varphi-P_{0})=\D(\partial\varphi)\cap H^{1}(\Gamma)$ and $(\hat\gamma_{0}u,\hat\gamma_{1}u)\in\partial\varphi\subset \D(\varphi)\times L^{2}(\Gamma)$. So, by elliptic regularity, 
$u\in H^{3/2}(\Omega)$. The proof is then concluded by \eqref{R} and \eqref{PHI1}.
\end{proof}
\begin{remark}
If $j:\RE\to (-\infty,+\infty]$ is a proper, convex, lower semicontinuous function, then
$$
\varphi:L^{2}(\Gamma)\to (-\infty,+\infty]\,,
\qquad 
\varphi(f):=\begin{cases}\int_{\Gamma}j(f(x))\,d\sigma(x) \,,& j(f)\in L^{1}(\Gamma)\\
+\infty\,,&\text{otherwise}\,,
\end{cases}
$$
is proper, convex and  lower semicontinuous and $(f,\tilde f)\in\partial\varphi$ if and only if 
$\tilde f(x)\in\partial j(f(x))$ for a.e. $x\in\Gamma$ (see \cite[Appendice I]{[Br1]}). If $|j(s)|\le c\, |s^{2}|$, then $\D(\varphi)=L^{2}(\Gamma)$ and so, by Proposition \ref{robin}, one obtains nonlinear, local Robin-type nonlinear boundary conditions,
$$
\D((-\Delta)_{\varphi})=\left\{u\in \D(\Delta^{\!\max})\cap H^{3/2}(\Omega): \gamma_{1}u(x)\in \partial j(\gamma_{0}u(x))\ \text{for a.e. $x\in\Gamma$} \right\}\,,
$$
and $(-\Delta)_{\varphi}$ belongs to the class of nonlinear maximal monotone operators studied in \cite[Section I.2]{[Br1]}. For such a kind of nonlinear local boundary conditions, by \cite[Th\'eor\`eme I.10]{[Br1]} (see also \cite[Proposition 2.9]{[Barb]}), one has $\D((-\Delta)_{\varphi})\subseteq H^{2}(\Omega)$. \par The situation $\D(\varphi)=L^{2}(\Gamma)$ is typical; for example $j(s)=b\, 1_{(0,+\infty)}(s)\,s^{2}$, $b>0$, gives the free boundary Cauchy problem (appearing in temperature control problems, see \cite[Chapter II]{[DuLi]}):
$$
\begin{cases}
\frac{\partial}{\partial t}u(t,x)=\Delta u(t,x)\,,&(t,x)\in(0,+\infty)\times\Omega\\
u(0,x)=u_{0}(x)\,,& x\in\Omega\\
\frac{\partial \,}{\partial n}u(t,x)=b\,u(t,x)\,,&(t,x)\in\Gamma(u)\\
\frac{\partial \,}{\partial n}u(t,x)=0\,,& (t,x)\notin \Gamma(u)\,,\\
\end{cases}
$$
where $\Gamma(u):=\{(t,x)\in(0,+\infty)\times\Gamma:u(t,x)>0\}$.
\end{remark}

\begin{remark}\label{lr}
If ${\mathfrak m}\not=\emptyset$, where
\begin{align*}
&\ {\mathfrak m}:=
\{f\in H^{1}(\Gamma):(f,P_{0}f)\in \partial\varphi\}
=
\{f\in H^{1}(\Gamma):\text{$f$ is a minimum point of $\varphi+\phi_0 $}\}
\,,
\end{align*} 
then, denoting by $S^{\varphi}_{t}$ the nonlinear semigroup of contractions generated by 
$(-\Delta)_{\varphi}$, by Remark \ref{Remark 4.6} one has
\begin{equation*}
\forall u\in\overline{\D((-\Delta)_{\varphi})}\,,\ \exists f_{u}\in{\mathfrak m}\ :\quad \text{w-}\lim_{t\to+\infty}S^{\varphi}_t(u)=u_{\infty}\,,
\end{equation*}
where $u_{\infty}=Kf_{u}$ is the unique harmonic function in $\Omega$ such that $ \gamma_{0}u_{\infty}=f_{u}$. If $\varphi$ is an even function then the above limit holds in strong sense.
\end{remark}
\begin{remark}\label{linear} In recent years there has been a renovated interest for the connection between theory of self-adjoint
extensions and boundary conditions for partial differential operators due to its re-formulation
in terms of Kre\u\i n's resolvent formula (see e.g. \cite{[BL1]}, \cite{[P08]}, \cite{[BMNW]}, \cite{[Gru]}, \cite{[BGW]}, \cite{[Mal]}, \cite{[GeMi]}, \cite{[BL2]} and
references therein). Proposition \ref{robin} provides a non-linear extension of such kind of results: if $\varphi(f)=\frac12\,\langle Bf,f\rangle_{L^{2}(\Gamma)}$, where $B$ is a symmetric bounded linear operator in $L^{2}(\Gamma)$, then one obtains the (non local) linear Robin-type  boundary conditions $\gamma_{1}u=B\,\gamma_{0}u$.
\end{remark}
\begin{remark}
Let $\varphi:L^{2}(\Gamma)\to[0,+\infty]$, 
 $\D(\varphi)\cap H^{1/2}(\Gamma)\not=\emptyset$, be a proper lower semicontinuous convex function and let us suppose that the corresponding $\Phi: L^{2}(\Omega)\to [0,+\infty]$ 
 given in \eqref{PHI1} is densely defined. Let us further suppose that
\begin{equation}\label{op}
\varphi(f\wedge g)+\varphi(f\vee g)\le \varphi(f)+\varphi(g)\,.
\end{equation}
Here $a\wedge b:=\min\{a,b\}$ and $a\vee b:=\max\{a,b\}$. Then, proceeding as in the linear case considered in \cite{[P12]}, by $\gamma_{0}(f\wedge g)=\gamma_{0}f\wedge \gamma_{0}g$ and $\gamma_{0}(f\vee g)=\gamma_{0}f\vee \gamma_{0}g$, one can check that 
$$
\Phi(u\wedge v)+\Phi(u\vee v)\le \Phi(u)+\Phi(v)\,.
$$
Thus, by \cite[Th\'eor\`eme 2.1]{[Bart]}, the contraction nonlinear semigroup $S_{t}^{\varphi}: L^{2}(\Omega)\to L^{2}(\Omega)$ generated by $(-\Delta)_{\varphi}:=\partial\Phi$ is order preserving, i.e.
$$
u,v\in L^{2}(\Omega)\,,\ u\le v\quad\Longrightarrow\quad \forall t\ge 0\,,\ S_{t}^{\varphi}(u)\le S_{t}^{\varphi}(v)\,.
$$ 
Similarly, if $\varphi$ has the property
\begin{equation}\label{cont}
\forall\alpha>0\,,\quad\varphi(g+ p_{\alpha}(f, g))+ \varphi(f - p_{\alpha}(f, g))
\le \varphi(f)+\varphi(g)\,,
\end{equation}
where 
$$
p_{\alpha}(f, g) := \frac12\,
((f - g +\alpha)_{+} -(f - g -\alpha)_{-})\,,\quad
$$
then, by \cite[Theorem 1.4]{[CG1]}, \cite[Theorem 3.6]{[CG2]}, the semigroup $S^{\varphi}_{t}$ is a contraction in $L^{\infty}(\Omega)$, i.e. 
$$
\forall t\ge 0\,,\forall u,v\in L^{\infty}(\Omega)\,,\quad  \|S_{t}^{\varphi}(u)-S_{t}^{\varphi}(v)\|_{L^{\infty}}\le \|u-v\|_{L^{\infty}}\,.
$$
In conclusion the nonlinear semigroup $S^{\varphi}_{t}$, $t\ge 0$, is Markovian whenever \eqref{op} and \eqref{cont} hold. Equivalently (similarly to the linear case, see \cite{[P12]}) , $S^{\varphi}_{t}$, $t\ge 0$, is a nonlinear Markovian semigroup in $L^{2}(\Omega)$ whenever $\partial\varphi$ generates a  nonlinear Markovian semigroup in $L^{2}(\Gamma)$.
\end{remark}

\end{subsection}
\begin{subsection}{Laplacians with nonlinear singular perturbations supported on $d$-sets}\label{E2} Here we follow an approach similar to the one provided (in a linear framework) in \cite[Example 3.6]{[P01]}. A closed Borel set $N\subset\RE^n$ is called a $d$-set, $0<d\le n$, if 
$$
\exists\, c_1,\,c_2>0\ :\ \forall\, x\in N,\ \forall\,r\in(0,1),\quad
c_1r^d\le\mu_d(B_r(x)\cap N)\le c_2r^d\ ,
$$
where $\mu_d$ is the $d$-dimensional Hausdorff measure and $B_r(x)$ is
the closed $n$-dimensional ball of radius
$r$ centered at 
the point $x$ (see \cite[Section 1.1, Chapter VIII]{[JoWa]}). 
Examples of $d$-sets for $d$ integer are finite unions of $d$-dimensional
Lipschitz sub-manifolds  and, in the not integer case, 
self-similar fractals of Hausdorff dimension $d$ (see \cite[Chapter II, Example 2]{[JoWa]}). 
\par
Now let $A_{\circ}=-\Delta_{D}:\D(\Delta_{D})\subseteq L^{2}(\Omega)\to L^{2}(\Omega)$ be the Dirichlet Laplacian as in the previous example and let 
$N\subset\Omega$, $N\cap \Gamma=\emptyset$,  be a compact $d$-set with $2<n-d<4$. 
(for simplicity of exposition we do not consider the case with lower co-dimension; this would require either a more involved definition of the trace spaces or a more regular set $N$).  Then we take $\tau=\gamma_{N}:=\tilde\gamma_{N}E|\D(\Delta_{D})$, where   
$$
E:H^{2}(\Omega)\to H^{2}(\RE^{d})
$$
is the extension map
and
$$
\tilde\gamma_N: H^2(\RE^{d})\to H^s(N)\,,\qquad s=2-\,\frac{n-d}{2}
$$
is the unique linear continuous and surjective map which coincides on smooth functions with the evaluation at the set $N$. We refer to \cite[Chapter 3, Theorems 1 and 3]{[JoWa]}, for the existence of the map $\tilde\gamma_N$. Here $H^{s}(N)$, $0<s<1$, is defined as the Hilbert space of functions $f\in L^2(N;\mu_N)$ such that $\|f\|_{H^{s}(N)}<+\infty$, where
$$
\|f\|^2_{H^{s}(N)}:=
\|f\|^2_{L^2(N;\mu_{N})}+\int_{|x-y|<1}
\frac{|f(x)-f(y)|^2}{|x-y|^{d+2s}}\,d\mu_N(x)\,d\mu_N(y)\,.
$$
Here $\mu_N$ denotes the restriction of the $d$-dimensional
Hausdorff measure $\mu_d$ to the set $N$. \par
Given $f\in H^{s}(N)$, let $f\delta_{N}\in H^{-2}(\Omega)$ be the distribution with compact support supp$(f
\delta_{N})=N$ defined by
$$
(f\delta_N,u)=\langle f,\gamma_N u\rangle_{H^{s}(N)}\,.
$$
Here $H^{-2}(\Omega)$ denotes the dual of $H^{2}(\Omega)$ with respect to the extension $(\cdot,\cdot)$ of the scalar product in $L^{2}(\Omega)$.
Given $\lambda\ge 0$, let $R^{D}_{\lambda}:=(-\Delta_{D}+\lambda)^{-1}$ and define 
$$
\tilde R^{D}_{\lambda}: H^{-2}(\Omega)\to L^{2}(\Omega)
$$
by 
$$
\langle\tilde R^{D}_{\lambda}\nu,u\rangle= (\nu, R^{D}_{\lambda}u)\,,\quad \nu\in H^{-2}(\Omega)\,,\ u\in L^{2}(\Omega)\,.
$$
Then
$$
G_\lambda: H^{s}(N)\to L^2(\RE^n)\,,\qquad
G_\lambda f:=\tilde R^{D}_{\lambda} (f\delta_N) \,.
$$ 
Therefore, given  any nonlinear maximal monotone relation $\Theta\subset H^{s}(N)\times H^{s}(N)$, by Theorem \ref{estensioni} one gets a nonlinear maximal monotone operator $(-\Delta)_{\Theta}$ defined by 
$$(-\Delta)_{\Theta}:\D((-\Delta)_\Theta)\subseteq L^{2}(\Omega)\to L^{2}(\Omega)\,,\quad 
(-\Delta)_{\Theta}\,u=-\Delta_{D} u_{0}\,,
$$
\begin{align*}
\D((-\Delta)_\Theta)
:=
\left\{u\in L^2(\Omega): u
=u_{0}+\tilde R^{D}_{0}(f_{u}\delta_{N}),\ u_{0}\in H^2(\Omega)\cap H_{0}^{1}(\Omega),\ (f_{u},\gamma_{N}u_{0})\in\Theta\right\}
\end{align*}
with nonlinear resolvent 
$$
((-\Delta)_{\Theta}+\lambda)^{-1}(u)=(-\Delta_{D}+\lambda)^{-1}u+\tilde R^{D}_{\lambda}(
(\Theta-\lambda\gamma_{N}\Delta_D^{-1}G_{\lambda})^{-1}(\gamma_{N}(-\Delta_{D}+\lambda)^{-1}u)\delta_{N})\,.
$$
Notice that $(-\Delta)_{\Theta}$ can be alternatively defined by
$$
(-\Delta)_{\Theta}(u):=-(\Delta u+f_{u}\delta_{N})
$$
and so, in the case $\alpha:=\Theta^{-1}$ is single-valued, one has
$$
(-\Delta)_{\Theta}(u)=-(\Delta u+
\alpha(\gamma_{N} u_{0})\delta_N)\,.
$$
If $\Theta=\partial\varphi$, where $\varphi:H^{s}(N)\to(-\infty,+\infty]$ is a proper lower semicontinuous function (notice that ran$(G_{0})\cap H_{0}^{1}(\Omega)=\{0\}$ whenever $n-d>1$), then 
$(-\Delta)_{\Theta}=\partial\Phi$, where
$$
 \Phi(u):=\begin{cases}\frac12\,\|(-\Delta_{D})^{\frac12}u_{0}\|^{2}
+\varphi(f_{u})&u\in\D(\Phi)\\
+\infty&\text{otherwise},
\end{cases}
$$
and
$$
\D(\Phi):=\{u\in L^{2}(\Omega):u=u_{0}+\tilde R^{D}_{0}(f_{u}\delta_{N})\,,\ u_{0}\in H_{0}^{1}(\Omega)
\,,\ \varphi(f_{u})<+\infty\}\,.
$$  
\begin{remark}
Examples \ref{E3} and \ref{E2} can be combined by taking $A_{\circ}=-\Delta_{D}$ and
$$
\tau:\D(\Delta_{D})\to H^{1/2}(\Gamma)\oplus H^{s}(N)\,,\quad 
\tau u:=\gamma_{1}u\oplus\gamma_{N}u\,.
$$
In this case Theorem \ref{estensioni} provides maximal monotone extensions describing Laplacians with nonlinear boundary conditions at $\Gamma$ and nonlinear singular perturbations supported at $N$.
\end{remark}
\end{subsection}

\end{document}